\documentclass[a4wide,12pt,reqno]{amsart}
\usepackage{a4wide,url}
\usepackage[utf8]{inputenc} 
\usepackage{amsmath,amsthm}
\numberwithin{equation}{section}
\usepackage{amssymb}
\newtheorem{trm}{Theorem}[section]
\newtheorem{prop}[trm]{Proposition}

\newtheorem{lm}[trm]{Lemma}

\theoremstyle{remark}

\newtheorem{rmk-pl}[rmk]{Remarks}

\newtheorem{ex}{Example}[section]
\theoremstyle{definition}
\newtheorem{dfn}{Definition}[section]

\def\R{\mathbb{R}}
\newcommand{\Zp}[1]{(\Z/p\Z)^{#1}}
\def\Wt{\widetilde{W}}

\def\lcm{\text{lcm}}

\def\1{\mathbb{1}}

\def\N{\mathbb{N}}
\def\Z{\mathbb{Z}}

\def\P{\mathcal{P}}

\def\Vol{\mathrm{Vol}}

\newcommand{\abs}[1]{\left\lvert #1 \right\rvert}
\newcommand{\nor}[1]{\left\lVert #1 \right\rVert}
\DeclareMathOperator{\E}{\mathbb{E}}

\DeclareMathOperator{\diam}{diam}

\title{A higher-dimensional Siegel-Walfisz theorem}
\author{Pierre-Yves Bienvenu}
\address{School of Mathematics, University of Bristol, Bristol BS8 1TW, United Kingdom}
\email{pb14917@bristol.ac.uk}
\date{\today}
\begin{document}
\maketitle
\begin{abstract}
The Green-Tao-Ziegler theorem provides asymptotics for the number of prime
tuples of the form $(\psi_1(n),\ldots,\psi_t(n))$ when
$n$ ranges among the integer vectors of a convex body $K\subset [-N,N]^d$ and $\Psi=(\psi_1,\ldots,\psi_t)$ is a system of 
affine-linear forms whose linear coefficients remain bounded (in terms of $N$).
In the $t=1$ case, the Siegel-Walfisz theorem
shows that the asymptotic still holds when the coefficients vary like a power of $\log N$. We prove a higher-dimensional (i.e. $t>1$) version of this fact. 

We provide natural examples where our theorem goes beyond the
one of Green and Tao, such as the count of arithmetic of progressions of step $\lfloor \log N\rfloor$ times a prime
in the primes up to $N$.
We also apply our theorem to the determination of asymptotics for the number of linear patterns in a dense subset of the primes, namely the primes $p$ for which $p-1$ is squarefree. To the best of our knowledge, this is the first such result
in dense subsets of primes save for congruence classes.

%
\end{abstract}

\section{Introduction}
The prime number theorem in arithmetic progressions says that
\begin{equation}
\label{PNTAP}
\sum_{n\leq N}\Lambda(qn+b)= 1_{(q,b)=1}\frac{q}{\varphi(q)}N+o(N),
\end{equation}
where $\Lambda$ is the von Mangoldt function and $\varphi$ the Euler totient function, while $b$ and $q$ are two fixed integers. The Siegel-Walfisz theorem affirms that the asymptotic \eqref{PNTAP}  holds uniformly in $b$ and $q$ in the regime where $q=q(N)$ is a function of $N$ satisfying $q=\log^{O(1)}N$ and $b$ varies as $N\log^{O(1)}N$.

Using the notation
$$
\Lambda_q(n)=\frac{q}{\varphi(q)}1_{(n,q)=1}
$$
for $n\in\Z$ or $n\in\Z/q\Z$ and multiplicativity, we can rewrite
equation \eqref{PNTAP} as
\begin{equation}
\label{PNTAP2}
\sum_{n\leq N}\Lambda(qn+b)= N(\prod_{p\mid q}\Lambda_p(b)+o(1)).
\end{equation}
A theorem of
Green and Tao \cite{GT2},  which we now state and which relies on two conjectures later resolved by Green, Tao and Ziegler \cite{MN,GI}, may be seen as a higher-dimensional analogue of
\eqref{PNTAP2}.
\begin{trm}
\label{GT}
Let $L$ be a constant and
$\Psi=(\psi_1,\ldots,\psi_{t}):\Z^d\rightarrow \Z^t$ be a system of affine-linear forms of which no two are affinely related. 
Suppose that the coefficients of the
linear part $\dot{\Psi}$ are bounded by $L$, while the constant
coefficients $\psi_i(0)$ satisfy $\abs{\psi_i(0)}\leq LN$.
Let $K\subset [-N,N]^d$ be a convex body. Then
\begin{equation*}
\sum_{n\in\Z^d\cap K}\prod_{i=1}^{t}\Lambda(\psi_i(n))=
\beta_{\infty}\prod_p\beta_p+o_{d,t,L}(N^d),
\end{equation*}
where
$$\beta_{\infty}=\Vol(K\cap\Psi^{-1}(\R_+^t))$$
and
$$\beta_p=\E_{a\in (\Z/p\Z)^d}\prod_{i=1}^t\Lambda_p(\psi_i(a)).$$
\end{trm}
Above we have used the notation
$$\E_{a\in A}f(a)=\frac{1}{\abs{A}}\sum_{a\in A}f(a)$$
to denote the averaging operator. We also agree that the letter $p$ is reserved
for primes, thus $\prod_p$ implicitly means $\prod_{p\in\P}$.
The factors $\beta_p$ are known as \emph{local factors}.

An interesting extension was obtained by the same authors, together with Ford and Konyagin \cite{GTFK}. They showed that Theorem \ref{GT} was still valid when the constant coefficients satisfied the less restrictive condition $\abs{\psi_i(0)}\leq N\log^C N$. This relaxed condition recently allowed Tao and Ziegler \cite[Theorem 1.3]{TZ} to obtain an improvement of the error term $o(N^d)$ to $o(\Vol(K))$ in the case where $K=[N]\times [M]^{d-1}$ with $M\gg N\log^{-O(1)}N$ and $\psi_i(n)=n_1+P_i(n_2,\ldots,n_{d})$ for some affine-linear forms $P_1,\ldots,P_t$ whose linear coefficients are bounded. 

Here we prove a further extension, allowing unbounded linear coefficients. The impetus for the work came from a discussion on Tao's blog \cite{blog}. Before stating our theorem,
we collect several definitions pertaining to systems of
affine-linear forms.
\begin{dfn}
\label{dfn:systems}
A system $\Psi=(\psi_1,\ldots,\psi_{t}):\Z^d\rightarrow \Z^t$ with $d,t\geq 2$ of affine-linear forms 
has \emph{finite complexity} if
no form is in the affine-linear span of another one.
It is called \emph{admissible}
if it is of finite complexity and if $\beta_p\neq 0$ for all $p$.
In particular, if $\Psi$ is admissible,
no prime $p$ divides all coefficients of any form $\psi_i$. 
A system $\Psi'=(\psi_{i_1},\ldots,\psi_{i_s})$ for some sequence
$1\leq i_1<\ldots<i_s\leq t$ is called a \emph{subsystem} of $\Psi$.

The system $\Psi$ is called \emph{bounded} when all its linear coefficients are bounded (in terms of the asymptotic parameter $N$) and unbounded otherwise. Its \emph{size at scale $(N,B)$}
is defined as
$$
\nor{\Psi}_{N,B}=\frac{1}{\log^BN}\left(\sum_{i\in[t]}\abs{\frac{\psi_i(0)}{N}}+\sum_{i\in[t],j\in[d]}\abs{\dot{\psi}_i(e_j)}\right)
$$
where $(e_1,\ldots,e_d)$ is the canonical basis of the lattice $\Z^d$ and $\dot{\psi}_i$ denotes the linear part of $\psi_i$.

Furthermore, a prime $p$ is called \emph{exceptional} for $\Psi$ (and we write $p\in P_{\Psi}$) if there exist $i\neq j$ such that $\psi_i$ and $\psi_j$ are affinely related modulo $p$. In particular, if a form $\psi_i$ is a nonzero constant modulo $p$, then $p$ is exceptional.
\end{dfn}
We highlight that our definition of exceptional prime is different (less restrictive) than that of Green and Tao \cite[Theorem D.3]{GT2} and that the size of a system differs from their definition by  the initial factor $\log^B N$.

We now check that $\prod_p\beta_p$ is still convergent in the setting where the coefficients are unbounded.
Thus the next lemma plays the role of \cite[Lemma 1.3]{GT2}.
\begin{lm}
\label{lm:convBetap}
Let $\Psi$ be a system of affine-linear forms.
 Then if $p$ is not exceptional,
$$
\beta_p=1+O_{d,t}(p^{-2}).
$$
In particular, if $\Psi$ is admissible, the product $\prod_p\beta_p$ is convergent and nonzero.
\end{lm}
\begin{proof}
If two forms $\psi_i$ and $\psi_j$ are not affinely related modulo $p$, then the probability
as $a$ ranges over $(\Z/p\Z)^d$ that they vanish simultaneously at $a$ is $p^{-2}$,
by elementary linear algebra (see for instance \cite[Proposition C.5]{bienv}). Inclusion-exclusion then yields $\beta_p=1+O(p^{-2})$ for $p$ unexceptional.

Now, if $\Psi$ is admissible, only finitely many primes are exceptional. Indeed, if $\psi_i$ and $\psi_j$ are affinely related modulo $p$ then
all the $2\times 2$ minors of the matrix $(\dot{\psi}_k(e_\ell))_{k\in\{i,j\},\ell\in[d]}$ are divisible by $p$ although 
at least one these minors has to be nonzero because
they are not affinely related as forms over $\Z$. Moreover,
the hypothesis of admissibility implies that $\beta_p\neq 0$ for every prime $p$, so that the product is convergent and nonzero.
\end{proof}

We now state our main theorem.
\begin{trm}
\label{mytrm}
Let $d,t$ be positive integers  and $A,B,L$ be positive constants.
Assume $\Psi=(\psi_1,\ldots,\psi_{t}):\Z^d\rightarrow \Z^t$ is an admissible system. 
Suppose that $\nor{\Psi}_{N,B}\leq L$ and that $K\subset [-N,N]^d$ is a convex body satisfying $\Vol (K)\gg N^d\log^{-A}N$ and $\Psi(K)\subset\R_+^t$. 
Then
\begin{equation}
\label{myequation}
\sum_{n\in\Z^d\cap K}\prod_{i=1}^{t}\Lambda(\psi_i(n))=
\Vol(K)\prod_p\beta_p(1+o_{d,t,A,B,L}(1)).
\end{equation}
\end{trm}
Unlike Theorem \ref{GT}, this theorem
is still meaningful when the convex body $K\subset [-N,N]^d$
satisfies $\Vol(K)=o(N^d)$ and applies even when the linear coefficients are of size polylogarithmic in $N$.

In some special cases, Theorem \ref{mytrm} follows easily
from the work of Green and Tao. In fact, Theorem \ref{GT} is a consequence of an asymptotic for the unbounded system $W\Psi+b$ where
$\Psi$ is a bounded system, $W=\prod_{p\leq w}p=\log^{1+o(1)}N$
and $b=(b_1,\ldots,b_t)\in[W]^t$ is a $t$-tuple of integers coprime to $W$.
More generally, an unbounded system $q\Psi+b$ 
with $q=\log^{O(1)}N$ 
and $\Psi$ bounded
is tractable, via an asymptotic for the system $\Wt\Psi+c$
where $\Wt=Wq$.\footnote{This remark was already exploited by the author in a previous paper \cite[Section 2.3]{bienv}.} By decomposing into residue classes, this 
extends to systems $\Psi$ such that
for each $j$, the coefficients $\dot{\psi}_i(e_j)$ are bounded multiples of a common coefficient $q_j$. We show an example, 
corresponding to the count of $k$-term progressions of primes
whose common difference is a multiple of $q$. We have
$$
\sum_{\substack{1\leq n,d\\ n+(k-1)qd\leq N}}\prod_{i=0}^{k-1}\Lambda(n+iqd)=
\sum_{a\in [q]}\sum_{\substack{1\leq n,d\\ n+(k-1)d\leq \frac{N-a}{q}}}
\prod_{i=0}^{k-1}\Lambda(q(n+id)+a)
$$
and are thus left with a system of the form $q\Psi+b$ with $\Psi$
bounded.

We now provide less immediate examples where Theorem \ref{mytrm}
applies. 
\begin{ex}
%
%
What is the proportion of arithmetic progressions
$n+d\N$ whose $q_1$th,$\ldots$,$q_k$th terms are all primes?
Assume that $q_i=\lfloor \log^i N\rfloor$.
The answer is given by
$$
\sum_{1\leq n,d\leq N}\prod_{i=1}^k\Lambda(n+q_id).
$$
For this system, the factors $\beta_p$ can be easily expressed,
using the notation $h(p)$ for the number of classes modulo
$p$ occupied by $q_1,\ldots,q_k$, as
$$
\beta_p=\left(\frac{p}{p-1}\right)^k\frac{(p-1)(1+p-h(p))}{p^2}.
$$
\end{ex}
\begin{ex}
We can also count $k$-terms arithmetic progressions of primes up to $N$ whose common difference is $q=\lfloor \log N\rfloor$ times a prime.
This time the sum to consider is
$$
\sum_{1\leq n\leq n+(k-1)qd\leq N}\Lambda(d)\prod_{i=0}^{k-1}\Lambda(n+iqd).
$$
To simplify the expression of the local factors, assume
$\prod_{p\leq k}p\mid q$. Then
$$\beta_p=
\left(\frac{p}{p-1}\right)^{k+1}\frac{1}{p^2}
\left\lbrace\begin{array}{cc}
(p-1)^2 & \text{ if } p\mid q\\
(p-1)(p-k) & \text{ if } p\nmid q.
\end{array}\right.
$$
\end{ex}
\begin{ex}
We provide the asymptotic count of solutions to linear equations in the shifted squarefree primes, that is, primes $p$ for
which $p-1$ is squarefree. As it is not a direct application,
we give the details in the final section.
\end{ex}

In view of the Siegel-Walfisz theorem \eqref{PNTAP2}, one may
hope to write
$$\sum_{n\in\Z^d\cap K}\prod_{i=1}^{t}\Lambda(\psi_i(n))=
\Vol(K)\left(\prod_p\beta_p+o_{d,t,A,B,L}(1)\right)$$
instead of \eqref{myequation}, but unfortunately our method does
not yield this. Such an estimate is genuinely stronger given that $\prod_p\beta_p$ might well tend to infinity with $N$ (and possibly the linear coefficients as well). This weaker bound is ultimately due to the ineffectiveness of the Gowers norm estimate
\cite[Theorem 7.2]{GT2}.

To prove Theorem \ref{mytrm}, we first get rid of the convex body by decomposing into reasonably small boxes,
so that the theorem simply needs to be proven on boxes. In this context, the variables all have the same range and
are independent of each other, which makes it possible, after
the introduction of the $W$-trick,  to prove a suitable von Neumann theorem.\footnote{The paper of Green and Tao also proceeds via a destruction of the convex body $K$ in Appendix C, but their method ceases to bear fruit as
soon as $\Vol(K)=o((\diam K)^d)$.} The latter eliminates all but one form, say $\psi_1$, thanks to a pseudorandom majorant. The next step is to equalize all coefficients of the chosen form $\psi_1$, by decomposing the ranges of
averaging into congruence classes. We are then left requiring a Gowers norm estimate which was proven by Green and Tao,
conditionally on two conjectures later fully resolved by
Green, Tao  and Ziegler \cite{MN,GI}.

We assume a certain amount of familiarity with the original arguments of Green and Tao \cite{GT2}. Wherever only minor changes are needed to accommodate the unbounded coefficients, 
we will
simply describe what must be changed in the existing arguments.

\textbf{Asymptotic notation.} The main asymptotic parameter throughout the paper is $N$, and all the standard asymptotic notation $o,O,\ll$ refers to the limit
as $N\rightarrow\infty$. There is an exception: whenever we discuss local factors, the limit is when $p\rightarrow\infty$, see
for instance Lemma \ref{lm:convBetap}. Several statements require, sometimes implicitly, that $N$ be large enough, which we always assume. Many other asymptotic parameters are defined in terms of $N$, such as $M,X,Z,Y$ so that the limit
as $N\rightarrow\infty$ is the same as, say, the limit
as $M\rightarrow\infty$. Indices may be added to symbols such as $O$, in which case they indicate upon which parameters the implied constant depends. 

\section{First reductions}

As in \cite[Theorem 4.1]{GT2}, we show that we can assume 
that the 
affine forms not only satisfy $\psi_i(n)\geq 0$ but actually $\psi_i(n)\geq N^{9/10}$.
We remark that the set $$\{n\in K\cap\Z^d : \exists i\in[t]\quad \psi_i(n)\leq N^{9/10}\}$$ 
contains only $O(N^{d-1/10})=o(\Vol(K)N^{-1/20})$ elements, so that we can replace $K$ by the convex body $K\cap\bigcap_{i=1}^t\psi_i^{-1}([N^{9/10},+\infty))$ in Theorem \ref{mytrm}. 
Moreover, prime powers are so sparse that we can restrict $\Lambda$ to primes, replacing
$\Lambda$ by $\Lambda'=1_\P\log$.
This leads to the first reduction.
\begin{prop}
Let $\Psi=(\psi_1,\ldots,\psi_{t}):\Z^d\rightarrow \Z^t$ be an admissible system. 
Suppose that $\nor{\Psi}_{N,B}\leq L$ and that $K\subset [-N,N]^d$ is a convex body satisfying $\Vol (K)\gg N^d\log^{-A}N$ and $\Psi(K)\subset([N^{9/10},+\infty))^t$. Then
\begin{equation*}
\sum_{n\in\Z^d\cap K}\prod_{i=1}^{t}\Lambda'(\psi_i(n))=
\beta_{\infty}\prod_p\beta_p\left(1+o_{d,t,A,B,L}(1)\right).
\end{equation*}
\end{prop}
We perform one more elementary reduction, namely we reduce to normal form; for the definition of this notion, we refer to \cite[Definition 4.2]{GT2}. 
The existence of a normal form extension $\Psi'$ for a system $\Psi$ follows from \cite[Lemma 4.4]{GT2}.
We inspect its proof to check that $\nor{\Psi}_{N,B}\leq L$
implies $\nor{\Psi'}_{N,B}\leq L'$ for some constant $L'$ depending only on $L,B,d,t$.
With respect to the paper \cite{GT2}, the change is
that the vector $f_{d+1},\ldots,f_{d'}$ introduced there will now be of size $O_L(\log^C N)$ for some constant $C=O_{B,d,t}(1)$. This is because these vectors are obtained by Cramer's formula, i.e. by computing determinants of matrices of bounded dimensions whose coefficients are $O_L(\log^B N)$. We claim this procedure reduces Theorem \ref{mytrm} to the following proposition, corresponding to \cite[Theorem 4.5]{GT2}.
\begin{prop}
\label{reduction}
Let $d,t$ be positive integers and $A,B,L$ be positive constants.
Let $\Psi=(\psi_1,\ldots,\psi_{t}):\Z^d\rightarrow \Z^t$ be an admissible system in normal form. 
Suppose that $\nor{\Psi}_{N,B}\leq L$ and that $K\subset [-N,N]^d$ is a convex body satisfying $\Vol (K)\gg N^d\log^{-A}N$ and $\Psi(K)\subset([N^{9/10},+\infty))^t$. Then
\begin{equation*}
\sum_{n\in\Z^d\cap K}\prod_{i=1}^{t}\Lambda(\psi_i(n))=
\beta_{\infty}\prod_p\beta_p\left(1+o_{d,t,A,B,L}(1)\right).
\end{equation*}
\end{prop}

We now check that the arguments following \cite[Theorem 4.5]{GT2} 
yield the deduction of Theorem \ref{mytrm}
from Theorem \ref{reduction}. Let $\Psi$ be an admissible system and $K$ a convex body satisfying the hypotheses of Theorem \ref{mytrm}.
Let $f_{d+1},\ldots,f_{d'}$ be the vectors and $\Psi'$ the system in normal form produced by the procedure given above. 
Suppose all $f_i$ satisfy $f_i=O(\log^C N)$.
Letting $M=N\log^{-C}N$, we define an auxiliary convex body $K'$  by
$$K'=\{(n,m_{d+1},\ldots,m_{d'})\in \R^d\times [-M,M]^{d'-d}\mid 
n+\sum_{i=d+1}^{d'}m_if_i\in K\}$$
which is included in $[-N',N']^{d'}$ where $N'=O(N)$. Moreover,
we still have $$\Vol(K')=\Vol(K)(2M)^{d'-d}\gg N'^d\log^{-D}N'$$
for some constant $D$.
Finally, the local factors are left unchanged by this operation. Thus if the system $\Psi$ is admissible, so is $\Psi'$, so that
Proposition \ref{reduction} can be applied to it, which concludes
the proof of the reduction.

\section{Reduction to the case of a box}
We show that the main theorem follows from the following very particular case.
\begin{prop}
\label{Boxes}
Let $d,t$ be positive integers and $A,B,L$ be positive constants.
Let $\Psi=(\psi_1,\ldots,\psi_{t}):\Z^d\rightarrow \Z^t$ be an admissible system in normal form. 
Suppose that $\nor{\Psi}_{M,B}\leq L$ and that $\Psi([M]^d)\subset([M^{9/10},+\infty))^t$.
Then
\begin{equation*}
\sum_{n\in[M]^d}\prod_{i=1}^{t}\Lambda(\psi_i(n))=
M^d\prod_p\beta_p\left(1+o_{d,t,A,B,L}(1)\right).
\end{equation*}
\end{prop}

\begin{proof}[Proof that Proposition \ref{Boxes} implies Proposition \ref{reduction}.]
Let $K\subset [-N,N]^d$ be a convex body satisfying $\Vol (K)\gg N^d\log^{-A}N$. Let $$K'=\{x\in K\mid d(x,\partial K)\geq N\log^{-A-1}N\}$$ and $$K''=\{x\in \R^d\mid d(x,K)\leq N\log^{-A-1}N\}.$$
These are two convex bodies.
The arguments from elementary convex geometry displayed in \cite[Appendix A]{GT2} allow one to infer that
$$
\Vol(K')=\Vol(K)+O(N^d\log^{-A-1}N)=\Vol(K)(1+o(1))
$$ 
and the same for $K''$.
Now let $M=N\log^{-A-1}N/\sqrt{d}$ and consider the grid $(M\Z)^d$. Let $\mathcal{B}=\{c+[M]^d\mid c\in J\}$ be the collection of boxes defined by this grid that are included in $K$
and $\mathcal{B'}=\{c+[M]^d\mid c\in J'\}$ be the collection of boxes defined by this grid that meet $K$. Note that \begin{equation}
\label{encadrement}
K'\subset \bigcup_{B\in \mathcal{B}}B\subset K\subset\bigcup_{B\in \mathcal{B'}}B\subset K'' .
\end{equation}
The first inclusion is because if a box $B$ from the grid meets $K'$, then it is included in $K$.

Now let $\Psi=(\psi_1,\ldots,\psi_{t}):\Z^d\rightarrow \Z^t$ be a system of affine-linear forms of finite complexity. 
Suppose that $\nor{\Psi}_{M,B}\leq L$ and that $\Psi([M]^d)\subset([M^{9/10},+\infty))^t$.
Then
$$
\sum_{B\in\mathcal{B}}\sum_{n\in B\cap\Z^d}\prod_{i\in[t]}\Lambda(\psi_i(n))\leq \sum_{n\in K\cap \Z^d}\prod_{i\in[t]}\Lambda(\psi_i(n))\leq \sum_{B\in\mathcal{B'}}\sum_{n\in B\cap\Z^d}\prod_{i\in[t]}\Lambda(\psi_i(n))
$$

Now if $B=c+[M]^d$ with $c\in\Z^d$, letting $\Psi_c=\Psi+\dot{\Psi}(c)$, we can write
$$
\sum_{n\in B\cap\Z^d}\prod_{i\in[t]}\Lambda(\psi_i(n))=
\sum_{n\in[M]^d}\prod_{i\in[t]}\Lambda(\psi_{c,i}(n)).
$$
We check that the system $\Psi_c$ satisfies $\nor{\Psi_c}_{M,C}=O(1)$ for some constant $C$. Indeed, the linear coefficients are unchanged, so still of size $O(\log^B N)=O(\log^B M)$, and the constant coefficients are of size
$O(N\log^B N)=O(M\log^{A+B+1}M)$. Thus $C=A+B+1$ is good enough.
Moreover  $\Psi_c([M]^d)\subset([N^{9/10},+\infty))^t$.
Thus we can apply Proposition \ref{Boxes}. We note that the $\beta_p$ it produces for a box $B=c+[M]^d$ is in fact independent of $c$, because the translation invariance of $\Z/p\Z$ allows one to write
$$
\E_{a\in(\Z/p\Z)^d}\prod_{i\in[t]}\Lambda_p(\psi_i(a)+\dot{\psi_i}(c))=\E_{a\in(\Z/p\Z)^d}\prod_{i\in[t]}\Lambda_p(\psi_i(a+c))=
\E_{a\in(\Z/p\Z)^d}\prod_{i\in[t]}\Lambda_p(\psi_i(a)).
$$
Finally,
$$
\abs{\mathcal{B}}M^d\prod_p\beta_p(1+o(1))\leq \sum_{n\in K\cap \Z^d}\prod_{i\in[t]}\Lambda(\psi_i(n))\leq\abs{\mathcal{B'}}M^d
\prod_p\beta_p(1+o(1))
$$
Because of the inclusions \eqref{encadrement}, we see that
$$
\Vol(K)(1+o(1))=\Vol(K')\leq \abs{\mathcal{B}}M^d\leq \abs{\mathcal{B'}}M^d\leq \Vol(K'')=\Vol(K)(1+o(1)).
$$
This completes the proof of Proposition \ref{reduction}.

\end{proof}
\section{The $W$-trick}
We perform the $W$-trick in the same spirit as \cite{GT2}, so as to eliminate biases modulo small primes. 
We introduce
$$
w=w(N)=\log\log N
$$
and $$
W=\prod_{p\leq w}p=\log^{1+o(1)}N.
$$
However, a simple one-dimensional example shows that one has to
adapt it to our situation where coefficients are unbounded.
Indeed, consider the system made of one form in one variable, namely
$n\mapsto qn+b$ with $q$ of size roughly $\log N$. The $W$-trick consists in writing
$$
\sum_{n\leq N}\Lambda(qn+b)=\sum_{a\in[W],(qa+b,W)=1}\frac{W}{\varphi(W)}\sum_{n\leq N/W}\frac{\varphi(W)}{W}\Lambda(Wqn+qa+b).
$$
But imposing that $(qa+b,W)=1$ does not ensure that the inner sum is $N/W(1+o(1))$, because $qa+b$ could well have a common factor greater than $w$ with $q$: when the coefficients are bounded, their factors are all less than $w$ for large enough $N$ but this is not the case any more in our setting. Moreover the
relevant average is not $W/\varphi(W)$ but $Wq/\varphi(Wq)$ which may be different if $q$ has prime factors larger than $w$.\footnote{Nevertheless, it is easy to check using Mertens' theorem that if $w=\log\log N$ and $q\leq \log^BN$ that
$$Wq/\varphi(Wq)=(1+o_B(1))W/\varphi(W).$$
}
This suggests that the coefficients of the system have to be taken into account to determine a suitable parameter $\Wt$ instead of $W$.

We fix once and for all an admissible system
$\Psi_1=(\psi_1,\ldots,\psi_{t}):\Z^{d_1}\rightarrow \Z^{t_1}$
in normal form
satisfying $\nor{\Psi_1}_{M,B}\leq L$ for some constants $B,L>0$.
Let 
\begin{equation}
\label{minors}
Q=\prod_{\substack{i\in[t_1],j\in[d_1]\\\dot{\psi}_i(e_j)\neq 0}}\dot{\psi}_i(e_j)\times
\prod_{\substack{1\leq i<k\leq t_1\\1\leq j<\ell\leq d_1\\
\dot{\psi}_i(e_j)\dot{\psi}_k(e_\ell)
-\dot{\psi}_i(e_\ell)\dot{\psi}_k(e_j)\neq 0
}}(\dot{\psi}_i(e_j)\dot{\psi}_k(e_\ell)
-\dot{\psi}_i(e_\ell)\dot{\psi}_k(e_j))
\end{equation}
be the product of
the nonzero minors of size 1 and 2 in the matrix $(\dot{\psi}_i(e_j))_{i,j}$; thus 
$Q=O_L(\log^{O_{d,t,B}(1)}N)$.
Moreover, if a prime $p$ is exceptional for $\Psi_1$, it must divide $Q$ (see the proof of Lemma \ref{lm:convBetap}).

We now introduce $\Wt=WQ=O(\log^{O(1)}N)$. The hypothesis that $\psi_i(n)>N^{9/10}$ means that if $\psi_i(n)$ is to be
a prime number, it has to be coprime to $\Wt$.
Writing $$\Lambda'_{\Wt,b}(n)=\frac{\varphi(\Wt)}{\Wt}\Lambda'(\Wt n+b)$$
we get
\begin{equation}
\label{WQtrick}
\sum_{n\in[M]^{d_1}}\prod_{i=1}^{t_1}\Lambda'(\psi_i(n))=
\sum_{\substack{a\in [\tilde{W}]^{d_1}\\
\forall i\in[t_1],(\psi_i(a),\widetilde{W})=1}}\left(\frac{\widetilde{W}}{\varphi(\widetilde{W})}\right)^{t_1}
\sum_{n\in [M/\widetilde{W}]^{d_1}}\prod_{i=1}^{t_1}
\Lambda'_{\widetilde{W},b_i(a)}
(\tilde{\psi}_i(n))+O(\log^{O(1)}N)
\end{equation}
where $\tilde{\psi}_i$ differs from $\psi_i$ only in the constant coefficient (by a multiple of $\Wt$) and $b_i(a)\in[\Wt]$ is the reduction mod $\Wt$ of $\psi_i(a)$.

Using the notion of subsystem introduced in Definition \ref{dfn:systems}, we reduce Proposition \ref{Boxes} to the following one.
\begin{prop}
\label{BoxesWtricked}
Let $\Psi_0=(\psi_1^0,\ldots,\psi_{t_0}^0):\Z^{d_0}\rightarrow \Z^{t_0}$ be a subsystem of $\Psi_1$. 
Suppose that $\Psi_0([M]^{d_0})\subset([M^{8/10},+\infty))^{t_0}$ and that $b_i\in[\widetilde{W}]$ is coprime to $\widetilde{W}$ for any $i\in[t_0]$. 
Then
\begin{equation}
\sum_{n\in [M/\Wt]^{d_0}}\prod_{i\in[t_0]}(\Lambda'_{\Wt,b_i}(\psi_i^0(n))-1)=o((M/\Wt)^{d_0})
\end{equation}
\end{prop}
We show how the reduction works, adapting the argument following
\cite[Proposition 5.1]{GT2}.

Applying successively the decomposition \eqref{WQtrick}, 
the trivial identity $x=x-1+1$, and Proposition
\ref{BoxesWtricked} to systems $\tilde{\Psi}$ where $\Psi$ is a subsystem of $\Psi_1$, we have
\begin{align*}
\sum_{n\in[M]^{d_1}}\prod_{i=1}^{t_1}\Lambda(\psi_i(n)) &=
\sum_{\substack{a\in [\Wt]^{d_1}\\
\forall i\in[t_1],(\psi_i(a),\widetilde{W})=1}}\left(\frac{\widetilde{W}}{\varphi(\widetilde{W})}\right)^{t_1}\left(\frac{M}{\widetilde{W}}\right)^{d_1}
(1+o(1))+O(\log^{O(1)}N)\\
&=M^{d_1}(1+o(1))\E_{a\in[\Wt]^{d_1}}\prod_{i\in[t_1]}\Lambda_{\Wt}(\psi_i(a))+O(\log^{O(1)} N).
\end{align*}
By the Chinese remainder theorem, and the fact that
$\Lambda_{p^k}(b)=\Lambda_p(b)$, we have
$$\E_{a\in[\Wt]^{d_1}}\prod_{i\in[t_1]}\Lambda_{\Wt}(\psi_i(a))
=\prod_{p\mid WQ}\E_{a\in\Zp{d_1}}\prod_{i\in[t_1]}\Lambda_p(\psi_i(a))
=\prod_{p\mid WQ}\beta_p.
$$
Moreover, if a prime $p$ does not divide $Q$, it is not
exceptional for $\Psi_1$. Then Lemma \ref{lm:convBetap} 
implies 
that
$\beta_p=1+O(p^{-2})$, so that
$$\prod_{p\nmid WQ}\beta_p=\prod_{p>w}(1+O(p^{-2}))=1+O(w^{-1}).$$
This concludes the reduction.

\section{Reduction to a Gowers norm estimate}
Write $X=M/\Wt$ and fix a system $\Psi_0$ and a tuple $b_1,\ldots,b_{t_0}$ satisfying the conditions of Proposition \ref{BoxesWtricked}. 
If $t_0=1$, this proposition follows directly from the one-dimensional Siegel-Walfisz theorem \eqref{PNTAP}, so we suppose
$t_0\geq 2$.
Let $Q_0$ be the product of $2\times 2$ minors for the system $\Psi_0$
as defined by equation \eqref{minors}. In particular, $Q_0\mid Q$.
We have to prove
$$
\sum_{n\in[X]^d}\prod_{i\in[t_0]}F_i(\psi_i(n))=o(X^d)
$$
for $F_i=\Lambda'_{\Wt,b_i}-1$.
At this point Green and Tao move to some cyclic group $\Z/N'\Z$ with $N'=O(X)$, but we cannot do this here without any wrap around, because of the large (unbounded) coefficients.

\subsection{A pseudorandom majorant}
Recall that we have fixed an admissible system
$$
\Psi_0=(\psi_i^0)_{i\in[t_0]} : \Z^{d_0}\longrightarrow\Z^{t_0}.
$$
We introduce the notion of a \emph{derived system}. This captures
the important properties of the systems that arise from repeated applications of the Cauchy-Schwarz inequality.
\begin{dfn}
A system $\Psi : \Z^d\rightarrow\Z^t$ of affine-linear forms is said to be \emph{derived} from $\Psi_0$ if
the following conditions are all satisfied:
\begin{itemize}
\item $d\leq 2d_0$;
\item $t\leq 2^{d_0}t_0$;
\item $\nor{\Psi}_{N,B}\ll \nor{\Psi_0}_{N,B}$;
\item any exceptional prime for $\Psi$ divides $Q_0$.
\end{itemize}
\end{dfn}
We now define pseudorandomness, based on the so-called linear forms condition; it is a fairly standard notion (see \cite[Section 6]{GT2}) but our definition
differs slightly to allow unbounded coefficients.
\begin{dfn}
We say that a function $\nu_Z : [Z]\rightarrow \R_+$ satisfies the
$\Psi_0$-\emph{linear forms condition}
if for any system $\Psi$
derived from $\Psi_0$
we have
$$
\E_{n\in [Z]^d}\prod_{i\in[t]}\nu(\psi_i(n))=1+o(1).
$$
We also say that $\nu$ is a $\Psi_0$-\emph{pseudorandom measure}.
\end{dfn}
The next proposition is about the existence of a pseudorandom
majorant for a $\Wt$-tricked von Mangoldt function.

\begin{prop}
\label{existencePRM}
For any integers $b_1,\ldots,b_{t_0}$ in $[\Wt]$ coprime to $\Wt$, for $Z\gg N\log^{-O(1)}N$, there exists a $\Psi_0$-pseudorandom measure $\nu$ on $[Z]$
such that
\begin{equation}
\label{Majorant}
1+\Lambda'_{\Wt,b_1}+\cdots+\Lambda'_{\Wt,b_{t_0}}\ll \nu
\end{equation}
on $[Z^{3/5},Z]$.
\end{prop}
The construction of the majorant, identical to \cite{GT2}, shall be explained in Section \ref{pseudorandom}.

\subsection{Generalised von Neumann theorem}
In spite of the impossibility to move to a cyclic group, we attempt to prove an analogue of
\cite[Proposition 7.1]{GT2}. Compared with the setting in a cyclic group, we cannot assume some linear coefficients to be 1, and
the range $[X]$ is not translation invariant. Recall that $\Psi_0
:\Z^{d_0}\rightarrow\Z^{t_0}$ is a fixed system of affine-linear
forms in $s$-normal form; thus without loss of generality,
write its first form as
$$
\psi_1(n_1,\ldots,n_{s+1},y)=q_1n_1+\cdots+q_{s+1}n_{s+1}+\psi_1(0,y)
$$
with $q_i\neq 0$ for all $i$ and $\prod_{j\in[s+1]}\dot{\psi_i}(e_j)=0$ for all $i>1$. We have dropped the exponent $0$ in the names of the form of the system $\Psi_0$ and shall always do so in the sequel.
We now state our variant of the von Neumann theorem. 

\begin{trm}
\label{GVNT}
Let $f_1,\ldots,f_{t_0} :\Z\rightarrow \R$ be functions and $\nu$ be a $\Psi_0$-pseudorandom measure such that $\abs{f_i}\leq \nu$ for all $i$. 
Then
$$
\abs{\E_{n\in [X]^d}\prod_{i\in[t_0]}f_i(\psi_i(n))}\leq \abs{\E_{y\in[X]^{d-s-1}}\E_{n^{(0)},n^{(1)}\in [X]^{s+1}}\prod_{\omega\in\{0,1\}^{s+1}}f_1(\sum_{i=1}^{s+1} q_in_i^{(\omega_i)}+\psi_1(0,y))}^{1/2^{s+1}}+o(1)
$$
\end{trm}
We adapt the proof of Proposition 7.1'' in \cite{GT2} and, for brevity,
use the notation from that proof without redefining it. 
With this notation, the left-hand side equals
\begin{equation}
\label{LHStobound}
\E_{y\in[X]^{d-s-1},x\in[X]^{s+1}}\prod_{B\subseteq[s+1]}F_{B,y}(x_B).
\end{equation}
We observe that 
$$
F_{[s+1],y}(x_{[s+1]})=f_1(\sum_{i\in[s+1]}q_ix_i+\psi(0,y)).
$$
We have the bounds $$\abs{F_{B,y}}\leq \nu_{B,y}.$$
The introduction of the functions $F_{B,y},\nu_{B,y} : [X]^B\rightarrow \N$ hides the arithmetic nature of the setting,
blurring away the difference between cyclic group and
intervals of integers.
Thus we can  use \cite[Corollary B.4]{GT2}, 
with $A=[s+1]$ and 
$X_\alpha=[X]$ for all $\alpha \in [s+1]$.
Hence we bound 
\eqref{LHStobound} by 
\begin{equation}
\label{afterCORB4}
\E_{y\in [X]^{d-s-1}}\nor{F_{[s+1],y}}_{\square (\nu_{[s+1],y})}
\prod_{B\subsetneq [s+1]}\nor{\nu_{B,y}}_{\square (\nu_{B,y})}^{2^{\abs{B}-(s+1)}},
\end{equation}
where we recall that
$$
\nor{F}_{\square (\nu_{B,y})}^{2^{\abs{B}}}=\E_{x^{(0)},x^{(1)}\in [X]^B}\prod_{\omega\in \{0,1\}^B}F(x^{(\omega)})
\prod_{C\subsetneq B}\nu_C((x_i^{(\omega_i)})_{i\in C},y).
$$

By Hölder's inequality, it suffices to show that
\begin{equation}
\label{norm}
\E_{y\in [X]^{d-s-1}}\nor{F_{[s+1],y}}_{\square (\nu_{[s+1],y})}^{2^{s+1}}= \E_{y\in [X]^{d-s-1}}\E_{n^{(0)},n^{(1)}\in [X]^{s+1}}\prod_{\omega\in\{0,1\}^{s+1}}f_1(\sum q_in_i^{(\omega_i)}+\psi_1(0,y))+o(1)
\end{equation}
and that
$$
\E_{y\in [X]^{d-s-1}}\nor{\nu_{B,y}}_{\square (\nu_{B,y})}^{2^{\abs{B}}}=1+o(1)
$$
for all non empty $B\subseteq [s+1]$. 
To prove the latter, expand the left-hand side as
\begin{equation}
\label{averageNu}
\E_{y\in [X]^{d-s-1}}\E_{n^{(0)},n^{(1)}\in [X]^B}\prod_{C\subseteq B}\prod_{i: \Omega(i)=C}\prod_{\omega\in\{0,1\}^C}\nu(\psi_i((n_j^{(\omega_j)})_{j\in\Omega(i)},y)).
\end{equation}
which is an expression involving the average of $\nu$ on a 
system $$\Psi=(\psi_{i,\omega})_{i\in[t_0],\omega\in\{0,1\}^{\Omega(i)}} : \Z^d\rightarrow\Z^t$$ 
of linear forms and it easy to check that $d\leq 2d_0$ and $t\leq 2^{d_0}t_0$. It is also obvious that $\nor{\Psi}_{M,B}\ll\nor{\Psi_0}_{M,B}$. Let the prime $p$ be
exceptional for $\Psi$ and let us check that it divides $Q_0$; this would mean that $\Psi$ is derived from $\Psi_0$ and
thus we could apply the linear forms condition. So let $\psi_{i,\omega}\neq \psi_{k,\alpha}$ be two forms that are affinely related modulo $p$. Then if $i\neq k$, we conclude that $\psi_i$ and $\psi_k$ are related and thus the prime is exceptional for $\Psi_0$, which implies that it divides $Q_0$. Otherwise $i=k$ and thus $\omega\neq \alpha$, in other
words there exists $j\in[d_0]$ such that $\dot{\psi_i}(e_j)\neq 0$ and $\omega_j\neq \alpha_j$. Thus $p$ must divide $\dot{\psi}_i(e_j)$
and hence also $Q_0$.
By applying the linear forms condition, the term \eqref{averageNu} is $1+o(1)$.

Let us look at \eqref{norm}.
At this point, Green and Tao use the translation invariance of $\Z/N'\Z$ to perform a change of variable which is not possible here, but we make do without it.
As the system is in normal form and $t\geq 2$, the form $\psi_1$ does not genuinely use all the variables. Indeed, the form $\psi_2$ must also have its set of $s+1$ variables that it is the only one to use fully, in particular $\psi_1$ does not use them all. Let us thus assume that $\psi_1$ only uses $x_1,\ldots,x_{d-k}$ with $k\geq 1$, which enables us, by a slight abuse
of notation, to regard $\psi_1$ as a map from $\Z^{d-k}$ to $\Z$. Upon expanding the norm, the left-hand side of \eqref{norm} becomes
\begin{equation*}
\label{expansion}
\E_{x^{(0)},x^{(1)}\in [X]^{s+1},y\in [X]^{d-k-s-1}}\prod_{\omega\in\{0,1\}^{s+1}}f_1(\sum_{i=1}^{s+1} q_ix_i^{(\omega_i)}+\psi_1(0,y))\E_{z\in [X]^{k}}\prod_{\omega\in\{0,1\}^{s+1}}\prod_{C\subsetneq [s+1]}\nu_{C,(y,z)}(x_C^{(\omega_C)})
\end{equation*}
where $(y,z)$ is  the vector in $\Z^{d-s-1}$ obtained by concatenating $y$ and $z$.
We want to replace the inner expectation over $z$, which is a function of $(x^{(0)},x^{(1)},y)$ of average 1, by 1.
To do that, by Cauchy-Schwarz, it is enough to prove
$$\E_{x^{(0)},x^{(1)}\in [X]^{s+1},y\in [X]^{d-k-s-1}}\prod_{\omega\in\{0,1\}^{s+1}}\nu(\sum_{i=1}^{s+1} q_ix_i^{(\omega_i)}+\psi_1(0,y))=1+o(1)=O(1)$$
which follows directly from the linear forms condition, and
$$
\E_{x^{(0)},x^{(1)}\in [X]^{s+1},y\in [X]^{d-k-s-1}}\prod_{\omega\in\{0,1\}^{s+1}}\nu(\sum_{i=1}^{s+1} q_ix_i^{(\omega_i)}+\psi_1(0,y))\abs{\E_zW(x,y,z)-1}^2=o(1),
$$
where $W(x,y,z)=\prod_{\epsilon\in\{0,1\}^{s+1}}\prod_{C\subsetneq [s+1]}\nu_{C,(y,z)}(x_C^{(\epsilon_C)})$.
This amounts to
$$
\E_{x^{(0)},x^{(1)}\in [X]^{s+1},y\in [X]^{d-k-s-1}}\prod_{\omega\in\{0,1\}^{s+1}}\nu(\sum_{i=1}^{s+1} q_ix_i^{(\omega_i)}+\psi_1(0,y))(\E_zW(x,y,z))^j=1+o(1)
$$
for $j=0,1,2$. 
Let us inspect the left-hand side in the most intricate case,
namely $j=2$. Upon expanding the square, we get an expectation over $x^{(0)},x^{(1)},y,z^{(0)},z^{(1)}$, thus the system is in at most $2d_0$ variables. There are $2^{s+1}$ forms arising from $\psi_1$ and at most $2^{s+2}(t-1)$ other forms, which means together at most $2^{d_0}t_0$ forms. Now the reasoning we used
to analyse the average \eqref{averageNu} also applies here and
yields that the system is derived from $\Psi_0$.
Thus the linear forms condition applies and equation \eqref{norm} is proven, hence also Theorem \ref{GVNT}.
\subsection{A Gowers-norm estimate}
Together with the existence of a pseudorandom majorant provided by Proposition \ref{Majorant},
Theorem \ref{GVNT} reduces Proposition \ref{BoxesWtricked} to the following.
\begin{prop}
Let $b\in[\Wt]$ be coprime to $\Wt$. Let $B>0$ and $d\in \N$ be constants. Suppose $q_1,\ldots,q_{d}$ are divisors of $Q$ satisfying $q_i=O(\log^BN)$ while $c=O(N\log^B N)$.
Then we have
\begin{equation}
\label{eq:pseudoGowers}
\E_{x^{(0)},x^{(1)}\in [X]^{d}}\prod_{\omega\in\{0,1\}^{d}}(\Lambda'_{\Wt ,b}(\sum_{i=1}^d q_ix_i^{(\omega_i)}+c)-1)=o(1).
\end{equation}
\end{prop}
The progress compared to Proposition \ref{BoxesWtricked} is that each variable $x_i^{(\epsilon)}$
for $i\in[d]$ and $\epsilon\in\{0,1\}$ is affected throughout the system
by one and the same coefficient $q_i$. We now attempt to transform
the system so that all variables have the same coefficient $Q'$;
the price we pay is that the variables will not have the same ranges any more.

To this effect, we introduce
$$
Q_i=\prod_{j\neq i}q_j
$$
and variables $n_i^{(\omega_i)},m_i^{(\omega_i)}$
such that $x_i^{(\omega_i)}=Q_in_i^{(\omega_i)}+m_i^{(\omega_i)}$.
Then the left-hand side of equation \eqref{eq:pseudoGowers} 
decomposes as
$$
\E_{m_i^{(\omega_i)}\in [q_i]}\E_{n_i^{(\omega_i)}\in [X/q_i]}
\prod_{\omega\in\{0,1\}^d} (\Lambda'_{\Wt,b}(\sum_{i=1}^d Q'n_i^{(\omega_i)}+q_im_i^{(\omega_i)}+c)-1)+o(1),
$$
where $Q'=q_iQ_i$ for any $i$.

We recognise the function
$$
n\mapsto F_a(n)=\frac{\varphi(\Wt)}{\Wt}\Lambda'(Q'\Wt{}n+a)=\Lambda'_{Q'\Wt,a},
$$
where the equality holds because $\prod_{p\mid Q'}p$ divides $Q$ and hence $\Wt$.
The parameters $a$ occurring are
$$
a_\omega=\Wt(\sum_{i=1}^d q_im_i^{(\omega_i)}+c)+b
$$
and given that $(b,\Wt)=1$, we also have $(a_\omega,\Wt)=1$ and finally $(a_\omega,\Wt{}Q')=1$.

We remark that for any tuple $a\in[\Wt{}Q']^{2^d}$ of integers coprime
to $\Wt{}Q'$, we can create a common $\Xi$-pseudorandom majorant for the functions $1+F_{a_\omega}$ where $\Xi=(\xi_\omega)_{\omega\in\{0,1\}^d}$ is defined by
$$
\xi_{\omega}=(n_1^{(0)},\ldots,n_d^{(0)},n_1^{(1)},\ldots,n_d^{(1)})
\longmapsto \sum_{i=1}^dn_i^{(\omega_i)}.
$$
In fact we
can rewrite Proposition 4.2 with $\Wt{}Q'$ instead of $\Wt$, because $Q'$ still satisfies $Q'=O(\log^{O(1)}N)$, 
a bound which 
allows us to control the effect of exceptional primes in Proposition \ref{GPY}.

Thus we are left to prove that
$$
\E_{n_i^{(\omega_i)}\in [X_i]}\prod_{\omega\in\{0,1\}^d} (F_{a_\omega}(\sum_{i=1}^d n_i^{(\omega_i)})-1)=o(1)
$$
where each $X_i$ satisfies $N\log^{-C}N\ll X_i\leq N$.
Letting $Z=\max_i X_i$ and $K=\prod_i[X_i]$, we have $K\subset [Z]^d$ and $\Vol(K)\gg Z^d\log^{-C'}Z$. Thus we can apply the same
reasoning as in Section 2 where we approximated such a convex body
by a set of small boxes of equal 
sides,\footnote{The reader might object that we then used the positivity of the function to average, which is not available here, but we 
can just as well use the majorant and the linear forms condition
to bound the contribution of the few boxes included in $K''$ but not in $K'$.}
and it suffices to prove that
\begin{equation}
\label{almostGowers}
\E_{n^{(0)},n^{(1)}\in[Y]^d}\prod_{\omega\in\{0,1\}^d}(F_{a_\omega}(\sum_{i=1}^dn_i^{(\omega_i)})-1)=o(1)
\end{equation}
for some $Y\gg N\log^{-D}N$.
Now that the linear forms have bounded coefficients (namely 0 and 1), there is no more objection to the use of Green-Tao's generalised von Neumann theorem \cite[Proposition 7.1]{GT2},
as long as the functions $1+\Lambda'_{Q'\Wt,a_\omega}$ are dominated
by a pseudorandom measure, in the sense of Green-Tao \cite[Definition 6.2]{GT2}. Green and Tao proved the existence
of such a majorant, except that they had $W$ instead of $Q'\Wt$,
but this makes no difference as $Q'\Wt$ is $w$-smooth and divisible by $W$. See Section 6 for a review of the construction of
a majorant and a closer scrutiny of the role of the
linear coefficients in the linear forms condition.
Thus equation \eqref{almostGowers} follows from the claim
\begin{equation}
\label{eq:claim}
\nor{F_a-1}_{U^k([Y'])}=
\nor{\Lambda'_{Q'\Wt,a}-1}_{U^k([Y])}=o(1)
\end{equation}
for any $a\in [Q'\Wt]$ coprime to $Q'\Wt$.
This is almost \cite[Proposition 7.2]{GT2}.
Compared with that proposition, we have $W'=Q'\Wt=O(\log^{O(1)}N)$ instead of $W$.
One can inspect attentively the remainder of the argument of Green and Tao to notice that the required properties of $W$ are
\begin{itemize}
\item that it be divisible by all primes $p\leq w$ for some function $w=w(N)$ tending to infinity;
\item that it be $O(\log^{O(1)}N)$; this is crucial when applying the Möbius-Nilsequence theorem \cite{MN}, which comes with a saving
of size an arbitrary power of $\log N$.
\end{itemize}
These properties are equally satisfied by $Q'\Wt$. Thus the claim \eqref{eq:claim} holds. To complete the proof of Theorem \ref{mytrm}, there remains only to prove Proposition \ref{existencePRM}, which we shall do in the next section.

\section{The linear forms condition}
\label{pseudorandom}
We recall the notation from \cite{GT2}
$$
\Lambda_{\chi,R}(n)=\log R\left(\sum_{d\mid n}\mu(d)\chi\left(\frac{\log d}{\log R}\right)\right)^2
$$
for $R=N^{\gamma}$ a small power of $N$, a smooth function $\chi$ supported on $[-1,1]$
satisfying $\chi(0)=1$ and $\int{\chi'^2}=1$. The function 
$\Lambda_{\chi,R}$ is positive and
if $n$ is a prime
larger than $R$, then $\Lambda_{\chi,R}(n)=\log R$,  so that
$\Lambda'\leq \gamma^{-1}\Lambda_{\chi,R}$ on $[N^{\gamma},N]$.

We need to extend the range of application of \cite[Theorem D.3]{GT2}. There it is stated only for forms whose constant coefficients are bounded, although it was then applied to
other natural systems such as $\Phi=W\Psi+c$, as the extension
was straightforward. In \cite[Appendix A]{GTFK}, the constant coefficients
are allowed to be as large $N^{1.01}$.

We claim that the estimate can be pushed further.
\begin{prop}
\label{GPY}
Let $L,B$ be positive constants and 
$\Psi=(\psi_1,\ldots,\psi_t)$ an admissible system of affine-linear forms satisfying $\nor{\Psi}_{Z,B}\leq L$. 
Let $P_\Psi$ be the set of exceptional primes and
$X=\sum_{p\in P_\Psi}p^{-1/2}$.
Then 
$$
\sum_{n\in [Z]^d}\prod_{i\in[t]}\Lambda_{\chi,R}(\psi_i(n))=Z^d\prod_p\beta_p(1+o(1)).
$$
\end{prop}

We now carefully analyse what needs to be changed in the proof of Theorem D.3 of \cite{GT2} when the linear coefficients
are of size up to $\log^B N$. As remarked in \cite[Appendix A]{GTFK}, the first place where the bound on the coefficients
is used is page 1833, where it is said that $\alpha(p,B)=O(1/p)$
for $p$ large enough. In fact, as we have assumed from the outset
that no form of the system $\psi_i$ is divisible by any prime $p$,
this is always the case. This also means that $\beta_p=1+O(1/p)$
for all $p$. The next moment where
Green and Tao invoke the size of the coefficients is 
to get the bound $\beta_p=1+O(p^{-2})$;
but in fact it is valid as soon as $p$ is not exceptional, no matter the size of the coefficients. 
As seen in the proof of Lemma \ref{lm:convBetap},
the set $P_\Psi$ of exceptional primes is finite but
its cardinality may increase to infinity with $Z$.
What we have seen implies that the asymptotic formula
$$
\sum_{n\in [Z]^d}\prod_{i\in[t]}\Lambda_{\chi,R}(\psi_i(n))=Z^d\prod_p\beta_p(1+ e^{O(X)}\log^{-1/20}R)
$$
is still valid, where $X=\sum_{p\in P_\Psi}p^{-1/2}$. It remains to bound $X$.

If $p\in P_{\Psi}$, as already seen in Lemma \ref{lm:convBetap}, $p$ divides
the parameter $Q=\log^{O(1)}N$ introduced in equation \eqref{minors}.
Letting $\omega(Q)=O(\log Q)=O(\log\log N)$ be the number of 
its prime factors, we have
$$
X=\sum_{p\in P_\Psi}p^{-1/2}
\leq \sum_{p\leq \omega(Q)}p^{-1/2}\leq \sum_{n\ll \log\log N}n^{-1/2}\ll \sqrt{\log\log N}.
$$
Thus $e^{O(X)}\ll\log^{1/30} N$ while $\log^{-1/20} R\ll \log^{-1/20} N$, which gives
$e^{O(X)}\log^{-1/20} R=o(1)$. This completes the verification of
Proposition \ref{GPY}.

From Proposition
\ref{GPY}, the proof of Proposition \ref{existencePRM}, that is, the construction of a majorant satisfying the adequate
linear forms condition, runs just as in the paper of Green and Tao.
We provide it here. Let
$b_1,\ldots,b_{t_0}$ be integers in $[\Wt]$ coprime to $\Wt$.
Let $Z$ be an asymptotic parameter satisfying $Z\gg N\log^{-A}N$ for some constant $A>0$. Then, writing
$$
\nu(n)=\frac{1}{t_0+1}\left(1+\frac{\varphi(\Wt)}{\Wt}\sum_{i\in[t_0]}\Lambda_{\chi,R}(\Wt{}n+b_i)\right),
$$
we have
the bound \eqref{Majorant} for $n\in [Z^{3/5},Z]$
if $\gamma <3/5$.
To show that $\nu$ is a $\Psi_0$-pseudorandom measure, it is enough to
check that for any system $\Psi : \Z^d\rightarrow\Z^t$ derived from $\Psi_0$, any $s\leq t$ and any sequence $1\leq j_1<\cdots<j_s\leq t$, we have
\begin{equation}
\label{unpluso}
\left(\frac{\varphi(\Wt)}{\Wt}\right)^s\E_{n\in [Z^d]}\prod_{i\in [s]}\Lambda_{\chi,R}(\Wt\psi_{j_i}(n)+b_{j_i})=1+o(1).
\end{equation}
To this aim, observe that the hypotheses ensure that the system
$\Phi=(\Wt\psi_{j_i}+b_{j_i})_{i\in[s]}$ is
is admissible. 
Moreover, because $\Wt=O(\log^{O(1)}N)$,
the bound $\nor{\Phi_{Z,B}}=O(1)$ holds for some $B=O(1)$.
So we can use Proposition
\ref{GPY}. For $p\mid\Wt$, the local factor is simply $(p/(p-1))^s$
while if $p\nmid \Wt$, the prime $p$ is not exceptional for $\Psi_0$ and hence for $\Phi$, which implies $\beta_p=1+O(p^{-2})$ by Lemma \ref{lm:convBetap}.
Thus
$$
\prod_p\beta_p=\left(\frac{\Wt}{\varphi(\Wt)}\right)^s\prod_{p>w}(1+O(p^{-2}))=\left(\frac{\Wt}{\varphi(\Wt)}\right)^s(1+O(w^{-1}).
$$
This compensates exactly for the factor $(\varphi(\Wt)/\Wt)^s$ and
finishes the proof of equation \eqref{unpluso}, hence also of Proposition \ref{existencePRM}, and finally of Theorem \ref{mytrm}.

In the next section, we give a nice application of our higher-dimensional Siegel-Walfisz theorem.

\section{Application to the primes $p$ such that $p-1$ is squarefree}

The set of primes $p$ such that $p-1$ is squarefree is a well-known dense subset of the primes of density $\sum_a\frac{\mu(a)}{\varphi(a^2)}=\prod_p(1-1/p(p-1))$; this is a theorem of Mirsky \cite{Mirsky}. As any dense subset of the primes, it contains arbitrarily long arithmetic progressions, by the Green-Tao theorem \cite{GT1}. However, no asymptotic was available
so far for the count of
$k$-term progressions in this set, nor in fact in any dense subset of the primes (except residue classes). As a consequence
of Theorem \ref{mytrm},
we now prove such an
asymptotic; in fact, we obtain an asymptotic for the number of solutions in this set of primes to any
finite complexity system of equations. 

For convenience, let
$F$ be the von Mangoldt function
restricted to the squarefree shifted primes, that is $F(n)=\Lambda(n+1)\mu^2(n)$. Also we denote by $\N$ the set of positive integers.
%
\begin{trm}
\label{appliSQF}
Let $\Psi : \Z^d\rightarrow\Z^t$ be a system of affine-linear forms
of finite complexity and $K\subset [-N,N]^d$ a convex body. Suppose that the linear coefficients are $O(1)$, the constants ones are $O(N)$ and that $\Psi(K)\subset \N^t$.
Then there exists a constant $C(\Psi)$ (possibly equal to 0) such
that
\begin{equation}
\label{eq:appliSQF}
\sum_{n\in K\cap\Z^d}\prod_{i\in[t]}F(\psi_i(n))=C(\Psi)\Vol(K)+o(N^d).
\end{equation}
\end{trm}
The constant $C(\Psi)$ will appear explicitly in the proof, but
its expression is unpleasant, so we do not give it here.
Throughout the proof of this theorem, we will need the notation
$$
\alpha_\Psi(k_1,\ldots,k_t)=\E_{a\in (\Z/m\Z)^d}\prod_{i\in[t]}1_{k_i\mid\psi_i(a)}
$$
where $m=\lcm(k_1,\ldots,k_t)$.
Elementary convex geometry reveals that this is the density of points of the lattice $\{n\in\Z^d : \forall i\in[t]\quad k_i\mid\psi_i(n)\}$ per unit volume, in the sense that
for any convex body $K\subset [-B,B]^d$, we have
that
\begin{equation}
\label{volPacking}
\sum_{n\in K\cap\Z^d}\prod_{i\in[t]}1_{d_i\mid\psi_i(n)}
=\Vol(K)\alpha_\Psi(d_1,\ldots,d_t)+O(B^{d-1}\lcm(d_1,\ldots,d_t)).
\end{equation}
This follows from simple volume packing arguments (see \cite[Appendix A]{GT2}, \cite[Appendix C]{polyprog}).
We now prove Theorem \ref{appliSQF}.
\begin{proof}
We insert the formula $\mu^2(n)=\sum_{a^2\mid n}\mu(a)$
in the left-hand side of \eqref{eq:appliSQF}. Thus
\begin{equation}
\label{eq:ais}
\sum_{n\in K\cap\Z^d}\prod_{i\in[t]}F(\psi_i(n))
=\sum_{(a_1,\ldots,a_t)\in\N^t}\prod_{i\in[t]}\mu(a_i)\sum_{\substack{n\in K\cap\Z^d\\ \forall i\in[t]\, a_i^2\mid\psi_i(n)}}
\prod_{i\in[t]}\Lambda(\psi_i(n)+1).
\end{equation}
Now for any $a=(a_1,\ldots,a_t)\in\N^t$, we introduce the set
$$
L_a=\{n\in\Z^d : \forall i\in[t]\quad a_i^2\mid\psi_i(n)\}.
$$
Fix an $a$ for which
$L_a\neq \emptyset$ and let $n_0\in L_a$. 
Then
$$
L_{a}=n_0+\bigcap_{i=1}^t\ker g_i
$$
where $g_i:\Z^d\rightarrow\prod_{i\in[t]}\Z/a_i^2\Z$ is the affine-linear map obtained by applying $\dot{\psi_i}$ and then reducing modulo $a_i^2$. So $L_a$ is an affine sublattice of full rank: indeed, its direction contains $\{\prod_ia_i^2e_1,\ldots,\prod_ia_i^2e_d\}$.
Incidentally, this means that we can suppose that $n_0$ satisfies
$n_0\cdot e_i \in[\prod_ia_i^2]$, in particular
$\nor{n_0}\leq d\log^{2dC}$.

As a lattice of full rank, the direction $\overrightarrow{L_a}$
of $L_a$ has a $\Z$-basis: there exist $f_1,\ldots,f_d$ such that 
$$
L_{a}=\{n_0+\sum_{i=1}^dm_if_i\mid (m_1,\ldots,m_d)\in\Z^d\}.
$$
Because of a theorem of Mahler, we can assume that $\nor{f_i}\leq i\lambda_i$ for $i=1,\ldots,d$, where
$\lambda_1\leq \cdots\leq \lambda_d$ are the successive minima of the lattice $\overrightarrow{L_a}$ with respect to the Euclidean unit ball. 
Let $R^a$ be the affine transformation of $\R^d$ defined by $R^a(0)=n_0$ and $\dot{R^a}(e_i)=f_i$ for each $i\in[d]$.
Note that $L_a\cap K=R^a(\Z^d\cap K_a)$ where $K_a$ is also a convex body.
For the notions of geometry of numbers alluded to here, see for instance the notes of Green \cite{geonumbers}.

Now if one of the $a_i$ is larger than $\log^C N$, then
$K_a$ is small. Indeed, the set
of $n\in K\cap\Z^d$ such that there exists $i\in [t]$ and $a_i>\log^C N$ satisfying $a_i^2\mid \psi_i(n)$ has $O(N^d\log^{-C}N)$ elements. 
This follows from equation \eqref{volPacking} combined with the bound
$\alpha_{\psi_i}(a_i^2)\ll a_i^{-2}$, obtained by multiplicativity, linear algebra (for instance \cite[Corollary C.4]{bienv}) and the fact that the coefficients of $\psi_i$ are bounded, and finally $\sum_{a>x}a^{-2}\ll x^{-1}$.
Bounding the contribution to the left-hand side of \eqref{eq:ais} of this exceptional set of $n\in K\cap\Z^d$ using $F\ll\log$, and supposing that $C\geq 2t$, we obtain 
\begin{align*}
\sum_{n\in K\cap\Z^d}\prod_{i\in[t]}F(\psi_i(n))
&=\sum_{\substack{n\in K\cap \Z^d\\ \forall i\in[t]\,\forall a>\log^C N\, a^2\nmid \psi_i(n)}}\prod_{i\in[t]}F(\psi_i(n))+
O(N^d\log^{-C/2}N)\\
&=\sum_{1\leq a_1,\ldots,a_t\leq \log^C N}\prod_{i\in[t]}\mu(a_i)\sum_{n\in K\cap L_a}
\prod_{i\in[t]}\Lambda(\psi_i(n)+1)+
O(N^d\log^{-C/2}N).
\end{align*}

For each $i\in[t]$, the map $\psi_i^a : L_a\rightarrow\Z$ defined by
$$
\psi_i^a(n)=\frac{\psi_i(n)}{a_i^2}
$$
is an affine map.  Then introduce $\phi_i^a=\psi_i^a\circ R^a$.
This defines a system $\Phi^a : \Z^d\rightarrow\Z^t$ of
affine-linear forms which is again of finite complexity. Thus the
inner sum in the left-hand side of equation \eqref{eq:ais} may be written as
\begin{equation}
\label{innersum}
\sum_{n\in K\cap L_a}
\prod_i\Lambda(\psi_i(n)+1)=
\sum_{m\in K_a\cap\Z^d}\prod_i\Lambda(a_i^2\phi_i^a(m)+1).
\end{equation}
We now apply Theorem \ref{mytrm} to the inner sum.
One can check that the linear coefficients of $\Phi_a$ have size
$O(\log^{O(1)}N)$. To do this, it is enough to examine the size of
the basis vectors $f_j$ of the lattice $\overrightarrow{L_a}$. Indeed,
$$
a_i^2\abs{\dot{\phi_i^a}(e_j)}=\abs{\dot{\psi_i}(f_j)}\leq \nor{\dot{\psi_i}}\nor{f_j}\ll\nor{f_j}.
$$
Moreover, the constant coefficients are $O(N)$.
As observed,
if $n_0\in L_a$, the lattice
$$
\{n_0+\sum_{i\in[d]}k_ia_i^2e_i\mid k\in\Z^d\}
$$
is a sublattice of $L_a$ and its determinant is
$\prod_i a_i^2\leq \log^{2dC}N$. Hence using Minkowski's second theorem, one finds that
$$
\prod_{i\in[d]} \nor{f_i}\leq d!\prod_{i\in[d]}\lambda_i\ll_d \abs{\det L_a}\leq \log^{2dC}N. 
$$
Similarly, we obtain the bound
$$
\Vol(K_{a})=\Vol(K)\det (R^a)^{-1}\geq\Vol(K)\log^{-2dC}N.
$$

Now Theorem \ref{mytrm} tells us that
the right-hand side of \eqref{innersum} is $\Vol(K_a)\prod_p\beta_p(1+o(1))$ as
soon as none of the local factors $\beta_p(a)$ corresponding to the system of the forms
$a_i^2\phi_i^a+1$ vanishes. 
Note that if any $\beta_p(a)$ is 0, then for all $m$ there exists $i\in[t]$ such that $p$ divides $a_i^2\phi_i^a(m)+1$. In this case,
the expression $a_i^2\phi_i^a(m)+1$ cannot be prime unless it is equal to $p$, which easily implies that
$$
\sum_{m\in K_a\cap\Z^d}\prod_i\Lambda(a_i^2\phi_i^a(m)+1)
=O(N^{d-1}\log^{O(1)}N).
$$
Moreover, equation \eqref{volPacking} reveals that
$$
\Vol(K_a)=\abs{K_a\cap\Z^d}+O(N^{d-1})=\abs{K\cap L_a}+O(N^{d-1})
=\Vol(K)\alpha_\Psi(a_1^2,\ldots,a_t^2)+O(N^{d-1}\log^{O(1)}N).
$$
Thus the left-hand side of equation \eqref{eq:appliSQF} equals
$$
\Vol(K)(1+o(1))\sum_{1\leq a_1,\ldots,a_t\leq \log^C N}\alpha_\Psi(a_1^2,\ldots,a_t^2)\prod_p\beta_p(a)\prod_{i\in[t]}\mu(a_i)+O(N^d\log^{-C/2}N).
$$
We claim the sum over $a$ is absolutely convergent.
To see this, first observe that
the exceptional primes for the system of the forms $a_i^2\phi_i^a+1$ are divisors of
$a_i^2$ or exceptional primes for the system $\Phi^a$;
in any case, they are divisors\footnote{See equation \eqref{minors} and the remark following it.} of a parameter $Q(a)=O(\prod_ia_i^{O(1)})$.
For all other primes, we have $\beta_p=1+O(p^{-2})$ by Lemma \ref{lm:convBetap}, so that
$$
\prod_p\beta_p(a)\ll\prod_{p\mid Q(a)}\beta_p(a)\leq \left(\frac{Q(a)}{\varphi(Q(a))}\right)^t\ll (\log\log Q(a))^t\ll
(\log\log\prod_ia_i)^t.
$$
Then note that the sum
$$
\sum_{a_1,\ldots,a_t}(\log\log\prod_ia_i)^t\alpha_\Psi(a_1^2,\ldots,a_t^2)
$$
is convergent. Indeed, we have the bound
$$
\alpha_{\Psi}(a_1^2,\ldots,a_t^2)=\prod_p\alpha_\Psi(p^{2v_p(a_1)},\ldots,p^{2v_p(a_t)})
\ll
\prod_pp^{-2v_p(\max_i a_i)}=\lcm(a_1,\ldots,a_t)^{-2},
$$
where the inequality holds because the forms
$\psi_i$ have bounded linear coefficients and
if $p\nmid\psi_i$, then $\alpha_{\psi_i}(p^k)\leq p^{-k}$ (this is linear algebra and Hensel's lemma, see \cite[Corollary C.4]{bienv}). The convergence then follows from a trivial bound for the number of $t$-tuples $a$ of prescribed least common multiple $k$, namely $\tau(k)^t$.
This convergence result implies that
$$\sum_{1\leq a_1,\ldots,a_t\leq \log^C N}\alpha_\Psi(a_1^2,\ldots,a_t^2)\prod_p\beta_p(a)\prod_{i\in[t]}\mu(a_i)
=C(\Psi)+o(1),
$$
where
$$
C(\Psi)=\sum_{(a_1,\ldots,a_t)\in\N^t}\alpha_\Psi(a_1^2,\ldots,a_t^2)\prod_p\beta_p(a)\prod_{i\in[t]}\mu(a_i).
$$ This concludes the proof.

\end{proof}


\end{document}